\documentclass[a4paper, 11pt, English]{article}
\usepackage[utf8]{inputenc}
\usepackage[T1]{fontenc}
\usepackage{graphicx}
\usepackage{stmaryrd}
\usepackage[a4paper]{geometry}
\usepackage{amsmath,amsfonts,amssymb,amsthm,epsfig,epstopdf,url,array}
\usepackage{rotating}
\usepackage[colorlinks=true,citecolor=red,linkcolor=blue,pdfpagetransition=Blinds]{hyperref}
\usepackage{cleveref}
\usepackage{nameref}
\usepackage{enumitem}
\usepackage{comment}
\Crefname{paragraph}{Section}{Sections}
\usepackage{fancyhdr}

\usepackage{fullpage}

\newcommand{\ensemblenombre}[1]{\mathbb{#1}}
\newcommand{\N}{\ensemblenombre{N}}
\newcommand{\Z}{\ensemblenombre{Z}}

\newcommand{\R}{} 
\renewcommand{\R}{\ensemblenombre{R}}
\newcommand{\C}{\ensemblenombre{C}}

\newcommand{\T}{\ensemblenombre{T}}
\newcommand{\D}{\ensemblenombre{D}}
\newcommand{\Pl}{\ensemblenombre{P}}

\newcommand{\norme}[1]{\left\lVert#1\right\rVert}

\newcommand{\Sp}{\mathrm{Sp}}

\newcommand{\dive}[1]{\mathrm{div}}

\theoremstyle{plain} 
\newtheorem{prop}{Proposition}[section] 
\newtheorem{theo}[prop]{Theorem}

\theoremstyle{definition}

\newtheorem{rmk}[prop]{Remark}

\makeatletter
\let\original@addcontentsline\addcontentsline
\newcommand{\dummy@addcontentsline}[3]{}
\newcommand{\DeactivateToc}{\let\addcontentsline\dummy@addcontentsline}
\newcommand{\ActivateToc}{\let\addcontentsline\original@addcontentsline}
\makeatother

\begin{document}

\title{Approximate controllability of the linearized Boussinesq system in a two dimensional channel}
\author{Kévin Le Balc'h}

\maketitle
\begin{abstract}
In this paper, we prove an approximate controllability result for the linearized Boussinesq system around a fluid at rest, in a two dimensional channel, when the control acts only on the temperature, through the upper boundary. This result can be seen as a first step to obtain an open loop stabilization result of the nonlinear Boussinesq system, in the spirit of the article \cite{CE19} by Chowdhury, Ervedoza concerning the Navier Stokes equations. The proof relies on the well-known Fattorini criterion, i.e. we show an unique continuation property for the adjoint system, by expanding the solution in Fourier series and using ordinary differential equations arguments. More precisely, we prove that the spectrum of the adjoint operator splits into two parts corresponding respectively to the Stokes eigenvalues and the Dirichlet Laplacian eigenvalues. Whereas the second part can be treated easily thanks to the well-known form of the eigenfunctions, the first part requires to show that a matrix of size three, depending analytically of the parameters of the problem, is “generically” invertible.
\end{abstract}
\small
\tableofcontents
\normalsize

\section{Introduction}
Let $ \Omega = \T \times (0,L)$, where $\T$ is the one-dimensional torus, identified with the interval $(0,2\pi)$ with periodic boundary conditions in the $x_1$ variable, and $L \in (0,+\infty)$. In this setting, we consider an incompressible fluid, with velocity field $u= u(t,x_1,x_2) = (u_1(t,x_1,x_2), u_2(t,x_1,x_2)) \in \R^2$ and pressure $p= p(t,x_1,x_2) \in \R$, satisfying the Navier-Stokes equation coupled through a source term with the temperature $\theta = \theta(t,x_1,x_2) \in \R$, satisfying a heat equation with a convection term. This corresponds to the Boussinesq approximation
\begin{equation}
\label{eq:boussinesq}
\left\{
\begin{array}{l l}
\partial_t u + (u \cdot \nabla) u - \nu  \Delta u + \nabla p = \theta e_2 &\mathrm{in}\ (0,\infty)\times\Omega,\\
\text{div}\ u = 0&\mathrm{in}\ (0,\infty)\times\Omega,\\
u(t,x_1,0) = u(t,x_1,L) = 0 &\mathrm{on}\ (0,\infty)\times\T,\\
\partial_t \theta - \alpha \Delta \theta + u \cdot \nabla \theta =0 &\mathrm{in}\ (0,\infty)\times\Omega,\\
\theta(t,x_1,0)= 0,\ \theta(t,x_1,L)= h(t,x_1)&\mathrm{on}\ (0,\infty)\times\T,\\
(u,\theta)(0,\cdot)=(u_0,\theta_0)& \mathrm{in}\ \Omega.
\end{array}
\right.
\end{equation}
Here, we have used the notation $e_2$, for the second vector of the canonical basis of $\R^2$, i.e. $e_2 = (0,1)^T$. Moreover, we assume that the viscosity coefficient $\nu$ of the fluid and the diffusion coefficient $\alpha$ of the temperature belong to $(0,+\infty)$. \\
\indent In the control system \eqref{eq:boussinesq}, at time $t \in (0,+\infty)$, $(u(t,\cdot), p(t,\cdot), \theta(t, \cdot))$ is the \textit{state} and $h(t,\cdot)$ is the \textit{boundary control}, acting only on the component $\theta$ through the upper boundary.\\
\begin{figure}
\label{fig1}
    \center
    \includegraphics[width=8cm]{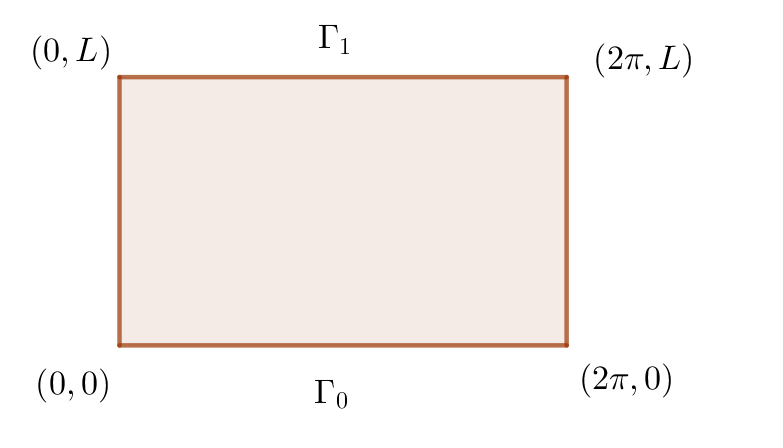}
    \caption{Two-dimensional channel $\Omega$}
\end{figure}
\indent Let us remark that a particular stationary solution of \eqref{eq:boussinesq} is $(u,p,\theta,h) = 0$, which corresponds to a fluid at rest. The linearization around $0$ of \eqref{eq:boussinesq} gives the following linear control system
\begin{equation}
\label{eq:boussinesqLin}
\left\{
\begin{array}{l l}
\partial_t u - \nu  \Delta u + \nabla p = \theta e_2 &\mathrm{in}\ (0,\infty)\times\Omega,\\
\text{div}\ u = 0&\mathrm{in}\ (0,\infty)\times\Omega,\\
u(t,x_1,0) = u(t,x_1,L) = 0 &\mathrm{on}\ (0,\infty)\times\T,\\
\partial_t \theta - \alpha \Delta \theta =0 &\mathrm{in}\ (0,\infty)\times\Omega,\\
\theta(t,x_1,0)= 0,\ \theta(t,x_1,L)= h(t,x_1)&\mathrm{on}\ (0,\infty)\times\T,\\
(u,\theta)(0,\cdot)=(u_0,\theta_0)& \mathrm{in}\ \Omega.
\end{array}
\right.
\end{equation}
We remark that the system \eqref{eq:boussinesqLin} couples a Stokes equation with a heat equation.\\
\indent If we expand $u$ into Fourier series, $u = \sum_{k \in \Z} u^k(t, x_2) e^{ikx_1}$, then the $0$-mode $u^0$ defined by 
\begin{equation}
u^0(t,x_2) = \int_{\T} u(t,x_1,x_2) dx_1 = \begin{pmatrix} u_{1}^0(t,x_2) \\
u_{2}^0(t,x_2) \end{pmatrix},
\end{equation}
satisfies the uncontrolled equation
\begin{equation}
\label{eq:boussinesqLin0}
\left\{
\begin{array}{l l}
\partial_t u_{1}^0 - \nu  \partial_{22} u_{1}^0 = 0&\mathrm{in}\ (0,\infty)\times\Omega,\\
u_{1}^0 (t,0) = u_{1}^0 (t,L) = 0 &\mathrm{in}\ (0,\infty),\\
u_{1}^0 (0,x_2)= \int_{\T} u_1^0(x_1,x_2) dx_1 &\mathrm{in}\ (0,L),\\
u_{2}^0(t,x_2) = 0 &\mathrm{in}\ (0,\infty)\times\Omega.
\end{array}
\right.
\end{equation}
Then, the goal of this article is to prove an approximate controllability result for the system \eqref{eq:boussinesqLin} for solutions, without $0$-mode, see \Cref{th:mainresult1} below.\\
\indent Roughly speaking, \eqref{eq:boussinesqLin} is approximately controllable because it is a cascade system. Indeed, the control $h$ acts on the component $\theta$, through the upper boundary, then $\theta$ indirectly controls the component $u_2$ thanks to the coupling term $\theta e_2$ then $u_2$ indirectly controls the component $u_1$ thanks to the incompressibility conditions $\partial_1 u_1 = - \partial_2 u_2$.
\section{Main result}
\subsection{Functional framework}
In this part, we recall some basic facts on the linearized system \eqref{eq:boussinesqLin} and give a modal description adapted to our setting. We will often deal with functions $u$ defined in $\Omega$, taking values in $\R^2$, hence belonging to functional spaces of the form $L^2(\Omega)^2$, or $H^1(\Omega)^2$. To simplify notations, we will simply denote by $\bold{L}^2(\Omega)$ and $\bold{H}^1(\Omega)$ the spaces $L^2(\Omega)^2$ and $H^1(\Omega)^2$ respectively.\\
\indent Recall that $\Omega = \T \times (0,L)$. For convenience, we define $\Gamma_0 = \{(x_1,0)\ ;\ x_1 \in \T\}$, $\Gamma_1 = \{(x_1,L)\ ;\ x_1 \in \T\}$ the lower and upper boundaries, and $\Gamma  = \Gamma_0 \cup \Gamma_1$, see Figure \ref{fig1}. We introduce the spaces
\begin{align}
\label{eq:defv0}\bold{V}^0(\Omega) &= \{ u =(u_1,u_2) \in \bold{L}^2(\Omega)\ ;\ \text{div}\ u = 0,\ \langle u \cdot n, 1 \rangle_{H^{-1/2}(\Gamma), H^{1/2}(\Gamma)} = 0 \},\\
\label{eq:defv1}\bold{V}_n^0(\Omega) &= \{ u \in \bold{V}^0(\Omega)\ ;\ u \cdot n = 0\ \text{on}\ \Gamma\},\\
\bold{X} &= \bold{V}_n^0(\Omega) \times L^2(\Omega),\\
\bold{V}^{1}(\Omega) & = \{u = (u_1,u_2) \in \bold{H}^1(\Omega)\ ;\ \text{div}\ u = 0\ \text{in}\ \Omega\},\\
\bold{V}_0^{1}(\Omega) & = \{u \in \bold{V}^{1}(\Omega)\ ;\ u(x_1,0)=u(x_1,L) = 0\ \text{for}\ x_1 \in \T\}.
\end{align}
The definitions of \eqref{eq:defv0} and \eqref{eq:defv1} are justified in \cite[Chapter IV, Section 3.2]{BF13}. We also introduce the Leray operator $\Pl$ as the orthogonal projection operator from $\bold{L}^2(\Omega)$ onto $\bold{V}_n^0(\Omega)$, see  \cite[Chapter IV, Section 3.3]{BF13}. The Boussinesq operator is then given by
\begin{equation}
\label{eq:defA}
A = A_0 + A_1 = \begin{pmatrix} 
\nu \Pl \Delta &(0)\\
(0) & \alpha \Delta \end{pmatrix} 
+ \begin{pmatrix} 0 & 0 & 0\\
0 & 0 & 1\\
0 & 0 & 0\end{pmatrix} = \begin{pmatrix} 
A_{S} & A_{C} \\
(0) & A_{L} \end{pmatrix}, 
\end{equation}
\begin{equation}
\ \text{with domain}\ D(A) = H^2(\Omega)^3 \cap \left(\bold{V}_0^1(\Omega) \times H_0^1(\Omega)\right)\ \text{on}\ \bold{X}.
\end{equation}
Indeed, the domain of the Stokes operator $A_S = \nu \Pl \Delta$ is given in \cite[Chapter IV, Section 5.2, Proposition IV.5.9]{BF13}, the domain of the Dirichlet Laplacian operator $A_L = \alpha \Delta$ is given in \cite[Section 9.6, Theorem 9.25]{Bre11} and $A_C \theta = (0, \theta)^T$ is a bounded operator, $A_C \in \mathcal{L}(L^2(\Omega);\bold{L}^2(\Omega))$.\\
\indent We now briefly describe the functional setting adapted to the linearized equations \eqref{eq:boussinesqLin}. To put the system in an abstract from, we introduce the Dirichlet operator $\D \in \mathcal{L}(L^2(\T) ; H^1(\Omega))$ defined by
\begin{equation}
\D h = w,\ \text{with}\ 
\left\{
\begin{array}{l l}
 - \alpha \Delta w =0 &\mathrm{in}\ \Omega,\\ 
w = 0&\mathrm{on}\ \Gamma_0,\\
w = h&\mathrm{on}\ \Gamma_1.
\end{array}
\right.
\end{equation}
Then the linearized equations of \eqref{eq:boussinesqLin} can be rewritten in the following abstract form
\begin{equation}
\label{eq:boussinesqabstract}
\left\{
\begin{array}{l l}
\Pl u' = \tilde{A_S} \Pl u + A_{C} \theta &\mathrm{for}\ t > 0,\\ 
(I - \Pl) u = 0 &\mathrm{for}\ t \geq 0,\\
\theta' = \tilde{A_L} \theta + (-\tilde{A_L}) \D h &\mathrm{for}\ t > 0,\\
(\Pl u, \theta)(0) = (\Pl u_0, \theta_0),
\end{array}
\right.
\end{equation}
where $\tilde{A_S}$ is the extension to the space $(D(A_S))'$ of the unbounded operator $A_S$ with domain $D(\tilde{A_S}) = \bold{V}_n^0(\Omega)$ and $\tilde{A_L}$ is the extension to the space $(D(A_L))'$ of the unbounded operator $A_L$ with domain $D(\tilde{A_L}) = H^1(\Omega)$, defined by the extrapolation method, see \cite[Chapter 10]{TW09}, \cite{Tri95}. Therefore, the control operator in \eqref{eq:boussinesqabstract} is admissible in the sense of \cite[Chapter 4, Section 4.2]{TW09} and reads as 
\begin{align}
B  :&\ L^2(\T) \longrightarrow D(A)'\\
 &\  h\qquad \longmapsto (0,0,(-\tilde{A_L})\D h).\label{eq:defB}
\end{align}
The study of the Cauchy problem for \eqref{eq:boussinesqabstract} is done in \Cref{sec:hypotheses}. For every $(u_0, \theta_0) \in \bold{X}$ and $h \in L^2(0,T;L^2(\T))$, equation \eqref{eq:boussinesqabstract} admits a unique weak solution in $C([0,T];\bold{X})$.
\subsection{Approximate controllability of the linearized system}
The goal of this section is to state our main result. As we have remarked in \eqref{eq:boussinesqLin0}, we cannot expect to control the $0$-mode of the solution. Therefore, we will work in the functional spaces
\begin{align}
\dot{\bold{V}_n^{0}}(\Omega)&= \left\{ u \in \bold{V}_n^{0}(\Omega) \ ;\ \int_{\T} u(x_1,x_2) dx_1 = 0,\ \forall x_2 \in (0,L)\right\} ,\\
\dot{L^2}(\Omega) &=\left\{ \theta \in  L^2(\Omega)\ ;\  \int_{\T} \theta(x_1,x_2) dx_1 = 0,\ \forall x_2 \in (0,L)\right\}  . 
\end{align}
Our main result is the following one.
\begin{theo}
\label{th:mainresult1}
Let $\nu \in (0,+\infty)$. There exists a countable subset $\mathcal{N}$ of $(0,\nu)$ such that for every diffusion coefficient $\alpha \in (0, \nu) \setminus \mathcal{N}$, for any positive time $T>0$, $\varepsilon >0$, initial data $(u_0, \theta_0) \in \dot{\bold{V}_n^{0}}(\Omega)\times \dot{L^2}(\Omega)$ and target $(u_1, \theta_1) \in \dot{\bold{V}_n^{0}}(\Omega)\times \dot{L^2}(\Omega)$, there exists a control $h \in L^2(0,T;L^2(\T))$ such that the associated solution $(u, \theta)$ of \eqref{eq:boussinesqLin} satisfies
\begin{equation}
\norme{(u,\theta)(T,\cdot) - (u_1,\theta_1)}_{\bold{V}_n^{0}(\Omega) \times L^2(\Omega)} \leq \varepsilon.
\end{equation}
\end{theo}
\begin{rmk}
Let us make some comments about \Cref{th:mainresult1}.
\begin{itemize}
\item We can give an explicit description of the subset $\mathcal{N}$, see the proof of \Cref{prop:uniquecontinuation}, \eqref{eq:defN}.
\item We do not know if the condition $\alpha \in (0, \nu) \setminus \mathcal{N}$ is a necessary condition for obtaining \Cref{th:mainresult1}. On the one hand, we conjecture that $\alpha < \nu$ is not necessary and is purely technical, see \Cref{rmk:restrictionalpha} below. On the other hand, we are convinced that for some diffusion coefficient $\alpha$, the unique continuation property proved in \Cref{prop:uniquecontinuation}, see below, can be violated. Indeed, such a phenomenon already appeared in the context of controllability of linear $2 \times 2$ parabolic systems, when the control only acts on a part of the boundary for one component, see \cite{FCGBdT10} or \cite[Section 4, Proposition 1 and Remark 11]{AKBGBdT10}.
\item By adding another control on the lower boundary for the temperature, i.e. by replacing the fifth equation of \eqref{eq:boussinesqLin}, by $\theta(t,x_1,0)= h_1(t,x_1),\ \theta(t,x_1,L)= h_2(t,x_1)$ on $(0,+\infty)\times\T$, of course \Cref{th:mainresult1} still holds. The fact that we can drop or not the same restriction on the diffusion coefficient $\alpha$ and preserves the approximate controllability is an interesting open question.
\item The approximate controllability of \eqref{eq:boussinesqLin} holds with two controls $u_2(t,x_1,L) = h_1(t,x_1)$ and $\theta(t, x_1,L) = h_2(t,x_1)$ on $(0,+\infty)\times\T$ without restriction on the diffusion coefficient $\alpha \in (0,+\infty)$, see \Cref{rmk:twocontrols} below.
\end{itemize}
\end{rmk}
From a physical point of view, \Cref{th:mainresult1} says that one is able to control the velocity and the temperature by acting only on the temperature though the upper boundary, i.e. by heating or cooling, to get closed to any state. As a consequence, this result is encouraging and is a first step to have practical applications in the near future to regulate the temperature in a room for instance.

\section{Proof of the approximate controllability}
\subsection{Hypotheses on the abstract control system}
\label{sec:hypotheses}
The goal of this section is to prove that the abstract control system \eqref{eq:boussinesqabstract} is well-posed and more precisely satisfies the hypotheses $(\mathcal{H}1), (\mathcal{H}2), (\mathcal{H}3), (\mathcal{H}4), (\mathcal{H}5)$ required by \cite{BT14}. In the following proposition, we use the notations of \cite{BT14}.
\begin{prop}
\label{prop:hypotheses}
Let $A$, $B$ be defined by \eqref{eq:defA} and \eqref{eq:defB} respectively.
\begin{equation}
\label{hyp1}
\text{The spectrum of}\ A\ \text{consists of isolated eigenvalues}\ (\lambda_j)\ \text{with finite algebraic multiplicity.}
\tag{i}
\end{equation}
\begin{equation}
\label{hyp2}
\text{The family of root vectors of}\ A\ \text{is complete in}\ \bold{X}.
\tag{ii}
\end{equation}
\begin{equation}
\label{hyp3}
\text{The operator}\ A=A_0+A_1\ \text{generates an analytic semigroup on}\ \bold{X}.
\tag{iii}
\end{equation}
\begin{equation}
\label{hyp4}
\tag{iv}
\text{The interpolation inequalities are satisfied}\ D((\mu_0 - A)^{\alpha}) = [\bold{X}, D(A)]_{\alpha},\ \forall \alpha \in [0,1].
\end{equation}
\begin{equation}
\label{hyp5}
\tag{v}
\text{The operator}\ B\ \text{is in}\ \mathcal{L}(L^2(\T) ; \bold{X}_{-1/2}),
\end{equation}
For any $(u_0, \theta_0) \in \bold{X}$, any $h \in L^2(0,T;L^2(\T))$, there exists a unique solution $(u, \theta)$ of \eqref{eq:boussinesqLin} such that 
\begin{equation}
\label{eq:regsolution}
(u, \theta) \in H^1(0,T; \bold{X}_{-1/2}) \cap L^2(0,T; \bold{V}_0^{1}(\Omega) \times H^1(\Omega)) \cap C([0,T]; \bold{X}).
\end{equation}
\end{prop}
\begin{proof}
\indent We see that $A$, defined in \eqref{eq:defA}, is a bounded perturbation on $\bold{X}$ of the self-adjoint operator $A_0$, which has compact resolvent. Then, we readily deduce that $A$ has also compact resolvent, so \eqref{hyp1} holds, see \cite[Section 6.3, Theorem 6.8]{Bre11}\\
\indent Moreover, according to Keldy's Theorem, see \cite[Remark 2.2]{BT14}, we have that \eqref{hyp2} holds. Indeed, $A$ is a bounded perturbation of the self-adjoint operator $A_0$, whose spectrum $(\mu_j)$ is exactly given by the spectrum of the Stokes operator and the Dirichlet Laplacian operator. So, we can easily check that
\begin{equation*}
\sum\limits_{j \in \N} \frac{1}{|\mu_j|} < + \infty,
\end{equation*}
see \cite[Lemma 2.2, Lemma 2.3]{CMR18} to get the asymptotic of the eigenvalues of the Stokes operator, see \Cref{prop:spectral} below to get to get the asymptotic of the eigenvalues of the Dirichlet Laplacian operator. \\
\indent By using the well-known fact that $A_0$ generates an analytic semigroup on $\bold{X}$ and $A_1$ is a bounded operator on $\bold{X}$, we deduce from \cite[Section 3.2, Theorem 2.1]{Paz83} that \eqref{hyp3} holds.\\
\indent By taking $\mu_0 > \sup_{j \in \N} \text{Re}(\lambda_j)$, it is easy to see that 
\eqref{hyp4} holds because $A$ is a bounded perturbation of a negative self-adjoint operator, see \cite[Section II.1.6]{BDPDM07}.\\
\indent It is also classical to prove that \eqref{hyp5} holds, see \cite[Section 10.7, Proposition 10.7.1]{TW09}.\\
\indent By the hypotheses \eqref{hyp1}, \eqref{hyp2}, \eqref{hyp3}, \eqref{hyp4} and \eqref{hyp5} and applying classical results on parabolic systems (see for instance \cite[Section 4.2, Proposition 4.2.5]{TW09}), we deduce the well-posedness result of \Cref{prop:hypotheses}. 
\end{proof}

\subsection{Spectral analysis of the linearized Boussinesq operator} 
The goal of this section is to give some spectral properties associated to the adjoint of the linearized Boussinesq operator. In this part, we only make the assumption $\alpha \neq \nu$.
\begin{prop}
\label{prop:spectral}
The spectrum of $A^{*}$ consists of isolated real eigenvalues with finite algebraic multiplicity. More precisely, we have 
\begin{equation}
\Sp(A^*) = \Sp_{S} \cup \Sp_{L},
\end{equation}
where
\begin{equation}
\label{eq:defspectrumStokes}
\Sp_S := \cup_{k \in \Z} \cup_{j \in \N} \{\lambda_{k_j}\}, \text{such that}\ \mu_{k_j} := \sqrt{ k^2 - \frac{\lambda_{k_j}}{\nu}} =  i\tilde{\mu_{k_j}} \ \text{satisfies}
\end{equation}
\begin{equation}
\label{eq:eigenvalueStokes}
- \sinh(kL) \sinh(\tilde{\mu_{k_j}} L) \tilde{\mu_{k_j}}^2 + 2k (1 - \cosh(kL) \cosh( \tilde{\mu_{k_j}} L)) \tilde{\mu_{k_j}} + k^2 \sinh(kL) \sinh(\tilde{\mu_{k_j}} L) = 0,
\end{equation}
and 
\begin{equation}
\label{eq:dirichletvalue}
\Sp_L := \cup_{k \in \Z} \cup_{j \in \N} \left\{-\alpha k ^2 - \frac{j^2 \pi^2}{L^2}\right\}.
\end{equation}
The eigenfunctions associated to the $k$-mode satisfy the following ordinary differential equation 
\begin{align}
\notag&- \xi_k^{(6)} + \left( \frac{\lambda}{\alpha}  + \frac{\lambda}{\nu} + 3 k^2\right) \xi_k^{(4)}- \left(\left(  \frac{\lambda}{\nu} + 2k^2\right)\left(\frac{\lambda }{\alpha}+ k^2\right) + k^2\left(\frac{\lambda}{\nu} + k^2\right)\right) \xi_k''\\
& + k^2\left(\frac{\lambda}{\nu}+ k^2\right)\left(\frac{\lambda}{\alpha}+ k^2\right) \xi_k = 0\qquad \text{in}\ (0,L),
 \label{eq:equationxi}
\end{align}
with boundary values
\begin{align}
&\xi_k(0) = \xi_k(L) = \xi_k''(0) =  \xi_k''(L)= \xi_k'''(0) - \left(\frac{\lambda}{\alpha}+k^2\right) \xi_k'(0) = \xi_k'''(L) - \left(\frac{\lambda}{\alpha}+k^2\right)   = 0.
\label{eq:boundaryvaluesxi}
\end{align}
Moreover, the eigenfunctions associated to $\Sp_L$ are given by
\begin{equation}
\label{eq:dirichleteigenfunction}
(\psi_{1,k_j},
\psi_{2,k_j},
q_{k_j},
\xi_{k_j})
(x_1,x_2) = \left(0,0,0,
 \sin\left(\frac{j \pi}{L} x_2\right)\right) e^{i k x_1}.
\end{equation}
\end{prop}
\begin{rmk}
\Cref{prop:spectral} essentially says that the spectrum of $A^*$ splits into two parts, corresponding respectively to the Stokes eigenvalues and the Dirichlet Laplacian eigenvalues. The Stokes eigenvalues are actually given by the eigenvalues of the operator $A_S$ defined in \eqref{eq:defA}. Let us mention that this spectrum has already been computed in \cite[Lemma 2.2, Item 4]{CMR18}.
\end{rmk}
\begin{proof}
Let $\lambda \in \C$, the equation $A^{*} \Phi = \lambda \Phi$ rewrites in its partial differential equation form
\begin{equation}
\label{eq:eigenfunction}
\left\{
\begin{array}{l l}
 \lambda \psi - \nu  \Delta \psi + \nabla q = 0 &\mathrm{in}\ \Omega,\\
\text{div}\ \psi = 0&\mathrm{in}\ \Omega,\\
\psi = 0 &\mathrm{in}\ \Gamma,\\
\lambda \xi - \alpha \Delta \xi = \psi_2 &\mathrm{in}\ \Omega,\\
\xi(t,x_1,0)= \xi(t,x_1,L)= 0&\mathrm{on}\ \T.
\end{array}
\right.
\end{equation}

We expand $(\Phi,q) = (\psi_1,\psi_2, \xi,q)$ into Fourier series in the $x_1$-variable
\begin{equation}
\label{eq:expandFourier}
(\Phi,q) =  (\psi_1,\psi_2, \theta,q)\ \text{with}\ 
\left\{
\begin{array}{l l}
\psi_1(x_1,x_2) = \sum_{k \in \Z} \psi_{1,k}(x_2) e^{ik x_1},& (x_1,x_2) \in \Omega,\\
\psi_2(x_1,x_2) = \sum_{k \in \Z} \psi_{2,k}(x_2) e^{ik x_1},& (x_1,x_2) \in \Omega,\\
\xi(x_1,x_2) = \sum_{k \in \Z} \xi_{k}(x_2) e^{ik x_1},& (x_1,x_2) \in \Omega,\\
q(x_1,x_2) = \sum_{k \in \Z} q_{k}(x_2) e^{ik x_1},& (x_1,x_2) \in \Omega.
\end{array}
\right.
\end{equation}
The eigenvalue problem \eqref{eq:eigenfunction} rewrites as follows
\begin{equation}
\label{eq:eigenvalueproblem}
\left\{
\begin{array}{l l}
(\lambda + \nu k^2) \psi_{1,k}(x_2) - \nu \psi_{1,k}''(x_2) + ik q_k(x_2) = 0 &\mathrm{in}\ (0,L),\\
(\lambda + \nu k^2) \psi_{2,k}(x_2) - \nu \psi_{2,k}''(x_2) + q_k'(x_2) = 0 &\mathrm{in}\ (0,L),\\
ik \psi_{1,k}(x_2) + \psi_{2,k}'(x_2) = 0 &\mathrm{in}\ (0,L),\\
\psi_{1,k}(0) = \psi_{1,k}(L) = \psi_{2,k}(0) = \psi_{2,k}(L) = 0,\\
(\lambda + \alpha k^2) \xi_k(x_2) - \alpha \xi_{k}''(x_2)= \psi_{2,k}(x_2) &\mathrm{in}\ (0,L),\\
\xi_k(0)= \xi_k(L) = 0.
\end{array}
\right.
\end{equation}
We distinguish two cases.\\
\indent \textit{Case 1: Dirichlet Laplacian eigenvalue, $\psi_{2,k} = 0$ in $(0,L)$.} From the third equation of \eqref{eq:eigenvalueproblem}, we find that $\psi_{1,k} = 0$ so $q_k = 0$ from the first equation of \eqref{eq:eigenvalueproblem}. Therefore, $\xi_k$ satisfies
\begin{equation*}
(\lambda + \alpha k^2) \xi_k(x_2) - \alpha \xi_{k}''(x_2)= 0,\ \mathrm{in}\ (0,L),\ 
\xi_k(0)= \xi_k(L) = 0.
\end{equation*}
This immediately gives the eigenfunction \eqref{eq:dirichleteigenfunction}, associated to the corresponding eigenvalue taken in \eqref{eq:dirichletvalue} and \eqref{eq:equationxi}, \eqref{eq:boundaryvaluesxi} trivially hold.\\
\indent \textit{Case 2: Stokes eigenvalue, $\psi_{2,k} \neq 0$.} By observing the first four equations of \eqref{eq:eigenfunction}, we remark that $\lambda$ is an eigenvalue of the (autoadjoint) Stokes operator. So $\lambda$ is real. From the third and the fourth equations of \eqref{eq:eigenvalueproblem}, we obtain that 
\begin{equation}
\label{eq:psi2Boundary}
\psi_{2,k}(0) = \psi_{2,k}(L) = \psi_{2,k}'(0) =  \psi_{2,k}'(L) = 0
\end{equation}
By using the first three equations of  \eqref{eq:eigenvalueproblem}, we obtain that $\psi_{2,k}$ satisfies the following ODE-equation 
\begin{equation}
\label{eq:equationpsi2}
\nu \psi_{2,k}^{(4)}(x_2) - (\lambda + 2 \nu k^2) \psi_{2,k}''(x_2) + k^2(\lambda + \nu k^2) \psi_{2,k}(x_2) = 0,\ x_2 \in (0,L).
\end{equation}
\indent If $\lambda \in [-\nu k^2,+\infty)$, then by multiplying \eqref{eq:equationpsi2} by $\psi_{2,k}$ and by integrating by parts using \eqref{eq:psi2Boundary}, we find
\begin{equation*}
\nu \int_0^L (\psi_{2,k}^{''})^2 +  (\lambda + 2 \nu k^2) \int_0^L (\psi_{2,k}')^2 + k^2(\lambda + \nu k^2)  \int_0^L \psi_{2,k}^2  = 0.
\end{equation*}
Then, we get $\psi_{2,k}' = 0$ in $(0,L)$ and by using another time \eqref{eq:psi2Boundary}, we obtain $\psi_{2,k} = 0$ in $(0,L)$ which contradicts our hypothesis $\psi_{2,k} = 0$. As a consequence, we will assume in the following that $\lambda \in (- \infty,-\nu k^2)$.\\
\indent From \eqref{eq:equationpsi2} and by using the fifth equation of \eqref{eq:eigenvalueproblem} and \eqref{eq:equationpsi2}, we obtain that $\xi_k$ satisfies the ordinary differential equation of order six \eqref{eq:equationxi} and the boundary conditions \eqref{eq:boundaryvaluesxi}.\\
\indent By denoting 
\begin{equation}
\label{eq:defmu1defmu2}
\mu_1 := \sqrt{k^2 +\frac{ \lambda}{\nu}},\ \mu_2 := \sqrt{k^2 +\frac{ \lambda}{\alpha}},
\end{equation}
the roots of the characteristic equation associated to the ODE \eqref{eq:equationxi} are
\[ \pm k\ \text{with multiplicity one}\ , \pm \mu_1\ \text{with multiplicity one}\ , \pm \mu_2\ \text{with multiplicity one}.\]
Here, we have used that $\alpha \neq \nu$.\\
\indent Then, we deduce that a fundamental basis of the ODE associated to \eqref{eq:equationxi} is 
\begin{equation}
\label{eq:defbasis}
(e^{\pm kx}, e^{\pm \mu_1 x}, e^{\pm \mu_2 x}).
\end{equation}
The eigenvalue problem associated to \eqref{eq:equationxi} and \eqref{eq:boundaryvaluesxi} is equivalent to
\begin{equation}
\label{eq:det0Bis}
\det(M(\mu))= 0,
\end{equation}
where $M(\mu)$ is equal to
\footnotesize
\begin{equation} 
\label{defmatrixdiffneq}
\begin{pmatrix}
1 & 1 & 1 & 1  & 1 & 1\\
e^{kL} & e^{-kL} & e^{\mu_1 L} & e^{-\mu_1 L} &  e^{\mu_2 L} & e^{-\mu_2 L}\\
k^2 & k^2 & \mu_1^2 & \mu_1^2 & \mu_2^2 & \mu_2^2\\
k^2 e^{kL} & k^2 e^{-kL} &\mu_1^2 e^{\mu_1 L}  &\mu_1^2 e^{-\mu_1 L} &\mu_2^2 e^{\mu_2 L}  &\mu_2^2 e^{-\mu_2 L}\\
 k^3 - k \mu_2^2 & k \mu_2^2 - k^3 & \mu_1^3 - \mu_1 \mu_2^2 & \mu_1 \mu_2^2 - \mu_1^3 & 0 & 0\\
e^{kL} (k^3 - k\mu_2^2) & e^{-kL} (k \mu_2^2 - k^3) & e^{\mu_1L} (\mu_1^3 - \mu_1\mu_2^2) & e^{-\mu_1 L} (\mu_1 \mu_2^2 - \mu_1^3) & 0 & 0
\end{pmatrix}.
\end{equation}
\normalsize A computation, performed in \textit{Maxima}, leads to 
\begin{align}
\det(M(\mu)) & = \left(\mu_1 e^{ \mu_1 L + k L} - k e^{ \mu_1 L + kL} - \mu_1 e^{\mu_1 L} - k e^{\mu_1 L} + e^{kL} \mu_1 - \mu_1 + k e^{kL} + k \right)\notag\\
&\left(\mu_1 e^{ \mu_1 L + k L} - k e^{ \mu_1 L + kL} + \mu_1 e^{\mu_1 L} + k e^{\mu_1 L} - e^{kL} \mu_1 - \mu_1 - k e^{kL} + k \right)\notag\\
& (\mu_2 -k )^2 (\mu_2+ k)^2 (\mu_2 - \mu_1)^2 (\mu_2 + \mu_1)^2 (e^{\mu_2 L} -1) (e^{\mu_2 L} +1) e^{- L \mu_2 - L \mu_1 - k L}. \label{eq:computationdet}
\end{align}
For obtaining \eqref{eq:det0Bis}, one of the factor of \eqref{eq:computationdet} is necessary equal to $0$.\\
\indent First, we remark that $\mu_2 \neq \pm k$ and $\mu_2 \neq \pm \mu_1$ because $\lambda \neq 0$ and $\alpha \neq \nu$. \\
\indent Secondly, the case $e^{\mu_2 L} = \pm 1$ corresponds to
\begin{equation}
\label{eq:findmu}
\sin(\tilde{\mu_2} L ) = 0,\ \text{i.e.}\ \tilde{\mu_2} =\frac{ j \pi}{L},\ l \in \Z,\ \text{with}\ \tilde{\mu_2} = i \mu_2.
\end{equation}
Then, by using the comatrix formula and by remarking that the last two columns of the matrix \eqref{defmatrixdiffneq} are equal in this case, we have
\begin{equation}
\label{eq:xiksinh}
\xi_k(x) = A \sinh(\mu_2 x ) = i A  \sin(\tilde{\mu_2} x),\ A \in \C.
\end{equation}
Now, by using the fifth equation of \eqref{eq:eigenvalueproblem} and \eqref{eq:xiksinh}, we obtain that $ \psi_{2,k}= - \alpha \xi_k'' + \alpha \mu_2^2 \xi_k =0$, which contradicts our assumption $\psi_{2,k} \neq 0$.\\
\indent We can check by a straightforward but lengthy computation that the product of the first two factors of \eqref{eq:computationdet} is equal to
\begin{equation*}
4 e^{L \mu_1 + L k} \left(  \sinh(kL) \sinh(\mu_1 L) \mu_1^2 -2 k (1 - \cosh(kL) \cosh( \mu_1 L)) \mu_1 + k^2 \sinh(kL) \sinh(\mu_1 L)\right),
\end{equation*}
then by setting $\mu_1 = i \tilde{\mu_1}$, we obtain the condition \eqref{eq:eigenvalueStokes}.
\end{proof}
\subsection{Fattorini-Hautus test: proof of \Cref{th:mainresult1}}
By using \cite[Theorem 1.1]{BT14} and \eqref{hyp1}, \eqref{hyp2}, \eqref{hyp3}, \eqref{hyp4} and \eqref{hyp5}, proving \Cref{th:mainresult1} reduces to show the following unique continuation property
\begin{equation}
\label{eq:HautusTest}
\text{If}\ A^* \Phi = \lambda \Phi\ \text{for some}\ \lambda \in \C\ \text{and if}\ B^* \Phi = 0,\ \text{then}\ \Phi = 0,
\end{equation}
where $\Phi$ is expanded in Fourier modes, without $0$-mode.\\
\indent The computation of $B^*\Phi$ yields to 
\begin{equation}
\label{eq:ComputeBstar}
B^* \Phi(x_1) = \frac{\partial \xi}{\partial x_2}(x_1,L) = \sum_{k \in \Z} \xi_{k}(L) e^{ik x_1},\ x_1 \in \T.
\end{equation}
By gathering \eqref{eq:HautusTest} and \eqref{eq:ComputeBstar}, we obtain that 
\begin{equation}
\label{eq:ComputeBstarBis}
\forall k \in \Z^{*},\ \xi_k'(L) = 0.
\end{equation}
\begin{prop}
\label{prop:uniquecontinuation}
There exists a countable subset $\mathcal{N}$ of $(0,\nu)$ such that for every $\alpha \in (0, \nu) \setminus \mathcal{N}$, assuming that $\Phi$ satisfies $A^{*} \Phi = \lambda \Phi$ for some $\lambda \in \R$ and $B^{*} \Phi = 0$ and expanding $(\Phi,q)$ as in \eqref{eq:expandFourier}, for all $k \in \Z \setminus \{0\}$, we have $\psi_{1,k} = \psi_{2,k} = \xi_k = q_k = 0$ everywhere in $(0,L)$.
\end{prop}
\begin{proof}
We prove the Fattorini-Hautus criterion by distinguishing two cases.\\
\indent \textit{Case 1: Dirichlet Laplacian eigenvalue.} If $\lambda \in \Sp_L$, we know the explicit form of the eigenfunction given by \eqref{eq:dirichleteigenfunction}, $\xi_k(x_2) = A \sin((j \pi /L) x_2)$ for some $A \in \C$, $j \in \N^{*}$. The observation \eqref{eq:ComputeBstarBis} immediately leads to $A = 0$.\\
\indent \textit{Case 2: Stokes eigenvalue.} We assume that $\lambda \in \Sp_{S}$.\\
\indent Let us multiply the first equation of \eqref{eq:equationxi} by $f$ where $f$ also satisfies the equation \eqref{eq:equationxi} and integrate by parts, using the boundary conditions \eqref{eq:boundaryvaluesxi} and \eqref{eq:ComputeBstarBis},
\begin{align}
&- \xi_k^{(5)}(L) f(L) + \xi_k^{(5)}(0) f(0) +  \xi_k^{(4)}(L) f'(L) - \xi_k^{(4)}(0) f'(0) + \xi_k^{(3)}(0) f''(0) - \xi_k'(0) f^{(4)}(0)\notag\\
& - \left( \frac{\lambda}{\alpha}  + \frac{\lambda}{\nu} + 3 k^2\right) \left( \xi_k^{(3)}(0) f(0)  + \xi_k'(0) f^{(2)}(0)\right)\notag\\
& + \left(\left(  \frac{\lambda}{\nu} + 2k^2\right)\left(\frac{\lambda }{\alpha}+ k^2\right) + k^2\left(\frac{\lambda}{\nu} + k^2\right)\right)\xi_k'(0) f(0) = 0.\label{eq:equationxif}
\end{align}
By recalling the notation \eqref{eq:defmu1defmu2} and the fundamental basis \eqref{eq:defbasis} of the ordinary differential equation \eqref{eq:equationxi}, we take $f$ of the following form
\begin{equation}
\label{eq:defMultiplyf}
f(x) = A \sinh(kx) + B \sinh(\mu_1 x) + C\sinh(\mu_2 x), \ (A,B,C) \in \R^3,
\end{equation}
We remark that $f(0) = f''(0) = f^{(4)}(0) = 0$, then we deduce that \eqref{eq:equationxif} becomes
\begin{equation}
\label{eq:equationxifbis}
- \xi_k^{(5)}(L) f(L) +  \xi_k^{(4)}(L) f'(L) - \xi_k^{(4)}(0) f'(0) = 0.
\end{equation}
\indent The next step is to prove that one can choose $f$ such that
\begin{equation*}
\label{eq:conditionf}
f'(0) = f(L) = 0,\ f'(L) \neq 0,
\end{equation*}
in order to obtain $\xi_k^{(4)}(L) = 0$ and one can also choose (another) $f$ such that
\begin{equation*}
\label{eq:conditionfBis}
f'(0) = f'(L) = 0,\ f(L) \neq 0,
\end{equation*}
in order to obtain $\xi_k^{(5)}(L) = 0$. Hence, $\xi_k$ satisfies the sixth order differential equation \eqref{eq:equationxi} with Cauchy data
\begin{equation*}
\xi_k(L) = \xi_k'(L) = \xi_k^{(2)}(L) = \xi_k^{(3)}(L) = \xi_k^{(4)}(L) = \xi_k^{(5)}(L) = 0,
\end{equation*}
which gives us $\xi_k = 0$ in $(0,L)$ by Cauchy-Lipschitz theorem, then as before $\psi_{2,k} = \psi_{1,k} = q_k = 0$ in $(0,L)$, which concludes the proof of \Cref{prop:uniquecontinuation}.\\
\indent It is then sufficient to check that the following matrix is invertible
\begin{equation}
\label{eq:defR}
R:= \begin{pmatrix}
k & \mu_1 & \mu_2\\
\sinh(kL) & \sinh(\mu_1 L) & \sinh(\mu_2 L)\\
k \cosh(kL) & \mu_1 \cosh(\mu_1L) & \mu_2 \cosh(\mu_2 L)
\end{pmatrix}.
\end{equation}
Indeed, by \eqref{eq:defMultiplyf}, we readily remark that 
\begin{equation*}
R \begin{pmatrix}
A\\
B\\
C
\end{pmatrix} = \begin{pmatrix}
f'(0)\\
f(L)\\
f'(L)
\end{pmatrix}.
\end{equation*}
Recall that from the proof of \Cref{prop:spectral}, we have $\lambda \in (- \infty,-\nu k^2)$, then $\mu_1 = i \tilde{\mu_1}$ and also $\mu_2 = i \tilde{\mu_2}$ because $\alpha < \nu$ with $(\tilde{\mu_1}, \tilde{\mu_2}) \in \R^2$. We compute the determinant
\begin{align*}
& F(\tilde{\mu_2}) := \det(R)=\\
&  \underbrace{\Big(\cos(\tilde{\mu_2} L) (k \sin(\tilde{\mu_1}L) -  \sinh(kL)\tilde{\mu_1}) + \sinh(kL) \tilde{\mu_1} \cos(\tilde{\mu_1}L)  -   k \cosh(kL) \sin(\tilde{\mu_1}L) \Big)}_{F^1(\tilde{\mu_2})} \tilde{\mu_2} \\
& + \tilde{\mu_1}\sin(\tilde{\mu_2} L)k \cosh(kL) - \tilde{\mu_1} \cos(\tilde{\mu_1}L) \sin(\tilde{\mu_2} L) k.
\end{align*}
We readily see that $F$ is an analytic function in a neighbourhood of infinity, which is not identically equal to zero. Indeed, by Rouché's theorem, for $\tilde{\mu_2}$ sufficiently large, the zeros of $F$ are exactly located in a neighbourhood of the zeros of $\tilde{\mu_2}F^1(\tilde{\mu_2})$. Moreover, $F^1$ is not identically equal to zero because the factor $k \sin(\tilde{\mu_1}L) -  \sinh(kL)\tilde{\mu_1} \neq 0$ because $\mu_1 \neq \pm k$. So, the set of zeros of $F$ is at most countable. Then, let us set
\begin{equation}
\label{eq:defN}
\mathcal{N} = \cup_{k \in \Z^*} \cup_{j \in \N} \left\{\alpha \in (0,\nu)\ ;\ F_{\lambda_{k_j}}\left(\sqrt{k^2 + \lambda_{k_j}/\alpha}\right) = 0\right\}.
\end{equation}
Let us mention that in \eqref{eq:defN}, we have underlined the dependence of $F$ in function of the eigenvalue $\lambda_{k_j}$ of the Stokes operator, defined in \eqref{eq:defspectrumStokes}. We have that $\mathcal{N}$ is a countable subset of $(0, \nu)$ because it is a countable union of countable sets. For every $\alpha \in (0,\nu) \setminus \mathcal{N}$, for all $\lambda_{k_j}$ eigenvalue of the Stokes operator, $F_{\lambda_{k_j}}(\tilde{\mu_2}) \neq 0$. So the matrix $R$, defined in \eqref{eq:defR}, is invertible for every $\alpha \in (0,\nu) \setminus \mathcal{N}$.
\end{proof}
\begin{rmk}
\label{rmk:restrictionalpha}
The restriction $\alpha < \nu$ seems to be purely technical in the previous proof. Indeed, we have used this assumption in order to have $\mu_2=i\tilde{\mu_2}$ purely imaginary, and to remark that $F(\tilde{\mu_2})$ is not identically equal to zero in a neighbourhood of infinity, i.e. $\alpha$ in a neighbourhood of zero. We cannot do the same in the limit $\alpha \rightarrow + \infty$ because in this case, we do not know how to prove that $\det(R)$, seen as an analytic function in $\mu_2$, is not identically equal to zero. Let us remark that $\alpha \rightarrow +\infty$ is equivalent to $\mu_2 \rightarrow \pm k$, so $\det(R) \rightarrow 0$ as $\alpha \rightarrow + \infty$.
\end{rmk}
\begin{rmk}
A more direct approach to prove \Cref{prop:uniquecontinuation} would be to explicitly compute $\xi_{k}'(L)$, where $\xi_k$ is defined as 
\begin{equation*}
\xi_k(x_2) = A_1 e^{k x_2} + A_{2} e^{-k x_2} + A_3 e^{\mu_1 x_2}  A_{4} e^{-\mu_1 x_2} + A_5 e^{\mu_2 x_2} + A_6 e^{-\mu_2 x_2},
\end{equation*}
with $(A_i)_{1 \leq i \leq 6}^T$ in the kernel of the matrix defined in \eqref{defmatrixdiffneq}. Then, $(A_i)_{1 \leq i \leq 6}$ can be computed explicitly by the comatrix formula for instance. Unfortunately, computations performed in \textit{Maxima} leads to an inextricable formula.
\end{rmk}
\begin{rmk}
\label{rmk:twocontrols}
If we add a control on the component $u_2$ on the upper boundary, i.e.  we consider \eqref{eq:boussinesqLin} with two controls $u_2(t,x_1,L) = h_1(t,x_1)$ and $\theta(t, x_1,L) = h_2(t,x_1)$, we prove that \Cref{prop:uniquecontinuation} holds then \Cref{th:mainresult1} holds without restriction on $\alpha \in (0, +\infty)$. Indeed, we can check that in this case, the observation \eqref{eq:ComputeBstarBis} becomes 
\begin{equation}
\label{eq:ComputeBstarTer}
\forall k \in \Z^{*},\ (q_k(L),\xi_k'(L)) = (0,0).
\end{equation}
By using the first, third, fifth equations of \eqref{eq:eigenvalueproblem} and the boundary conditions \eqref{eq:boundaryvaluesxi}, we obtain that $\xi_k$ satisfies
\begin{equation*}
\xi_k(L) = \xi_k'(L) = \xi_k^{(2)}(L) = \xi_k^{(3)}(L) = \xi_k^{(5)}(L) = 0.
\end{equation*}
So, \eqref{eq:equationxifbis} becomes
\begin{equation*}
\xi_k^{(4)}(L) f'(L) - \xi_k^{(4)}(0) f'(0) = 0.
\end{equation*}
Then, if we take $f(x) = \sinh(kx) - \frac{k}{\mu_1} \sinh(\mu_1 x)$, we check that $f$ satisfies $f'(0) = 0$ and $f'(L) \neq 0$ because $\mu_1 \neq \pm k$. This gives $\xi_k^{(4)}(L) = 0$ and we conclude as before.\\
\indent Another proof could be to use directly the unique continuation property of the Stokes system with the observation $q_k(L) = 0$ for every $k \in \Z^{*}$, see \cite[Proposition 2.1]{CE19}.
\end{rmk}

\section{Perspectives and related references}

\subsection{Null-controllability of the Boussinesq system}

A natural perspective that could be addressed in the future is the null-controllability of the linearized Boussinesq system \eqref{eq:boussinesqLin}. In this context, the techniques employed in the article \cite{CMR18}, which establishes the null-controllability of the incompressible Stokes equation, would be useful. The difficulty comes from the lengthy computations, due to the eigenvalue problem \eqref{eq:equationxi} and \eqref{eq:boundaryvaluesxi} associated to the ordinary differential equation of order six, instead of an ordinary differential equation of order four that appears for the Stokes problem, see \cite[Equation (2.4)]{CMR18}. Let us mention that the null-controllability of Boussinesq system with few internal controls, even in the nonlinear case, has been studied in the following articles \cite{FCGIP06}, \cite{GBGP09}, \cite{FCS12} and \cite{Car12}.

\subsection{Stabilization around a Poiseuille flow} 

We remark that another particular stationary solution of \eqref{eq:boussinesq} is given by a Poiseuille flow $(u_1, u_2, p, \theta, h) = (C x_2 (x_2-L), 0,0,0,0)$. As in \cite{Mun12}, we could try to get some exponential stabilization result for the linearized system around the previous stationary state, by a finite-dimensional feedback control acting on the temperature, through the upper boundary $\{x_2 = L\}$. Let us mention that the boundary stabilization of the Boussinesq system around unstable states, with feedback controllers for both velocity and temperature, has been studied in \cite{BT11}, \cite{Bad12} and more recently in \cite{RRR19}.

\subsection{Open loop-stabilization of the nonlinear Boussinesq system}
In the spirit of \cite{CE19}, a natural perspective would be to obtain some local open-loop stabilization result for the nonlinear Boussinesq system \eqref{eq:boussinesq} at any given decay rate. By looking at \eqref{eq:boussinesqLin0}, we see that any stabilization strategy based on the linearized system \eqref{eq:boussinesqLin} will fail to stabilize the system \eqref{eq:boussinesq} at a rate lower than $-\nu \pi^2/L^2$, corresponding to the first eigenvalue $\lambda_{0,1} = - \nu \pi^2/L^2$ of the operator $\nu \partial_{22}$ with Dirichlet boundary conditions on $(0,L)$.\\
\indent Let us explain briefly the approach of the authors in our context. It is based on a power series expansion, already used in \cite{CC04} in the context of KdV equation and \cite{BC06} for Schrödinger equation (see also \cite[Chapter 8]{Cor07} for an introduction to this method). We assume that the controlled solution $(u, \theta)$ and its control $h$ can be expanded as
\begin{equation*}
u = \varepsilon \alpha + \varepsilon^2 \beta,\ \theta = \varepsilon \theta_1 + \varepsilon^2 \theta_2,\ p = \varepsilon p_1 + \varepsilon^2 p_2,\ h= \varepsilon h_1 + \varepsilon^2 h_2,
\end{equation*}
for some $\varepsilon>0$ small enough, where $(\alpha, \beta)$, $(\theta_1, \theta_2)$, $(p_1, p_2)$, $(g_1,g_2)$ are all of order one. This allows to look for $(h_1,h_2)$ such that $(\alpha, \beta, \theta_1, \theta_2)$ 
\begin{equation}
\label{eq:boussinesqLinOrdre1}
\left\{
\begin{array}{l l}
\partial_t \alpha - \nu  \Delta \alpha + \nabla p_1 = \theta_1 e_2 &\mathrm{in}\ (0,\infty)\times\Omega,\\
\text{div}\ \alpha = 0&\mathrm{in}\ (0,\infty)\times\Omega,\\
\alpha = 0 &\mathrm{in}\ (0,\infty)\times\Gamma,\\
\partial_t \theta_1 - \alpha \Delta \theta_1 =0 &\mathrm{in}\ (0,\infty)\times\Omega,\\
\theta_1(t,x_1,0)= 0,\ \theta_1(t,x_1,L)= h_1(t,x_1)&\mathrm{on}\ (0,\infty)\times\T,\\
(\alpha,\theta_1)(0,\cdot)=(\alpha^0,\theta_1^0)& \mathrm{in}\ \Omega,
\end{array}
\right.
\end{equation}
\begin{equation}
\label{eq:boussinesqQuadOrdre2}
\left\{
\begin{array}{l l}
\partial_t \beta - \nu  \Delta \beta + \nabla p_2 = - (\alpha + \varepsilon \beta) \cdot \nabla (\alpha + \varepsilon \beta )+  \theta_2 e_2 &\mathrm{in}\ (0,\infty)\times\Omega,\\
\text{div}\ \beta = 0&\mathrm{in}\ (0,\infty)\times\Omega,\\
\beta = 0 &\mathrm{in}\ (0,\infty)\times\Gamma,\\
\partial_t \theta_2 - \alpha \Delta \theta_2 = - (\alpha + \varepsilon \beta) \cdot \nabla (\theta_1 + \varepsilon \theta_2) &\mathrm{in}\ (0,\infty)\times\Omega,\\
\theta_2(t,x_1,0)= 0,\ \theta_2(t,x_1,L)= h_2(t,x_1)&\mathrm{on}\ (0,\infty)\times\T,\\
(\beta,\theta_2)(0,\cdot)=(\beta^0,\theta_2^0)& \mathrm{in}\ \Omega,
\end{array}
\right.
\end{equation}
is stable and decays exponentially at rate $-\omega_0$.\\
\indent As the control function $h_1$ cannot act on the $0$-mode of $(\alpha,\theta_1)$, see \eqref{eq:boussinesqLin0}, we would put the unstable undetectable part of component of $(u,\theta)$ in the $(\beta, \theta_2)$ part. Our construction would therefore use the nonlinear term in \eqref{eq:boussinesqQuadOrdre2} to indirectly control the projection of $\beta$ on the unstable undetectable part.\\

\textbf{Acknowledgements:} I thank Sylvain Ervedoza and Marius Tucsnak for interesting discussions on this article. I also thank Karine Beauchard for reading a preliminary draft of this article.

\bibliographystyle{alpha}
\bibliography{UniqueContinuationBoussinesq}

\end{document}